\documentclass[12pt]{article}
\usepackage{amssymb,amsfonts,amsthm,amsmath,graphicx,fullpage,multirow,longtable,pst-barcode,caption2}
\newcommand\lws{\textbf{Ia}}
\newcommand\rws{\textbf{Ib}}
\newcommand\tbt{\textbf{II}}

\newcommand\twt{\textbf{IV}}

\newtheorem{thm}{Theorem}
\newtheorem*{crl}{Corollary}
\theoremstyle{remark}
\newtheorem{rmk}{Remark}
\theoremstyle{definition}
\newtheorem{dfn}{Definition}
\title{The Decomposition Algorithm of Skew-symmetrizable Exchange Matrices}
\author{Weiwen Gu\\
\small A521 Wells Hall\\
\small Michigan State University\\
\small East Lansing, MI, 48824\\
\small guweiwen@msu.edu}
\date{}
\begin{document}
\maketitle
\begin{abstract}
Some skew-symmetrizable integer exchange matrices are associated to ideal (tagged) triangulations of marked bordered surfaces. These exchange matrices admits unfoldings to skew-symmetric matrices. We develop an combinatorial algorithm that determines if a given skew-symmetrizable matrix is of such type. This algorithm generalizes the one in \cite{WG}. As a corollary, we use this algorithm to determine if a given skew-symmetrizable matrix has finite mutation type.
\end{abstract}
\section{Introduction}

With some triangulations of surfaces invariant under finite group symmetries, we associate finite oriented multi-graph without
loops and $2$-cycles. Such graphs are called \textit{quivers}. Positive integer weights are assigned to quivers. Such graphs are also associated with matrices with integer entries. We call such graphs or their associated matrices \textit{$s$-decomposable} (see Definition \ref{sdecom}.) Quiver \textit{mutation} is defined in \cite{FST}. The collection of mutation-equivalent quivers to a given quiver $G$ is called the mutation class of $G$. We say a quiver is \textit{mutation finite} or has \textit{finite mutation type} if its mutation class is finite.

A quiver or its associated skew-symmetric matrix is said to be block-decomposable if the quiver can be obtained by combining pieces of graphs isomorphic to six types of quivers, called elementary blocks, by a particular way of gluing (see Definition \ref{gluingrules}). It is proved in \cite{FST1} that a quiver has finite mutation type if and only if it is either block-decomposable or is of one of the 11 exceptional types. It is also proved in \cite{FST2} that a quiver is block-decomposable if and only if it is the associated adjacency graph of an ideal triangulations of a bordered surface with marked points.

An $n\times n$ integer matrix $B$ is said to be \textit{skew-symmetrizable} if there exists an $n\times n$ integer diagonal matrix $D$ such that $BD$ is skew-symmetric. Mutation and mutation class are also defined for skew-symetrizable matrices. Our goal is to establish an combinatorial algorithm that determines if a given oriented graph whose edges are equipped with integer weights is $s$-decomposable, thus providing a tool to find if its associated skew-symmetrizable matrix has finite mutation type. Each skew-symmetrizable exchange matrix is associated to a diagram with oriented edges equipped with positive integer weights. The notion of $s$-decomposability is a generalization of \textit{block-decomposability} for diagrams (see Definition \ref{sdecom} and Table \ref{BOU}). A skew-symmetrizable exchange matrix is said to be $s$-decomposable if it can be obtained by gluing both elementary blocks and 7 additional blocks by the rules in Def. \ref{gluingrules} and Def. \ref{sdecom}.

In \cite{WG}, we provided an algorithm linear in the size of quiver $G$ that determines if $G$ is block-decomposable. As a corollary, we obtained for any skew-symmetric integer matrix $B$, an algorithm linear in the size of $B$ determining if $B$ has finite mutation type. The algorithm we describe in this article is a generalization of the one in \cite{WG}. This paper is inspired by \cite{FST}, in which the authors generalize the result of \cite{FST1} to skew-symmetrizable \textit{exchange matrices}. The following results are proved in \cite{FST}:
 \begin{enumerate}
 \item There is a one-to-one correspondence between $s$-decomposable skew-symmetrizable graphs with fixed block decomposition and ideal tagged triangulations of marked bordered surfaces with fixed tuple of conjugate pairs of edges. Conjugate edges are two edges inside a digon (or monogon) with one of them tagged and the other untagged.
 \item A skew-symmetrizable $n\times n$ matrix, $n\geq3$, that is not skew-symmetric, has finite mutation class if and only if diagram is either $s$-decomposable or mutation-equivalent to one of seven exceptional types.
 \item Any $s$-decomposable diagram admits an \textit{unfolding} (see Definition \ref{unfoldingdef}) to a diagram associated to ideal tagged triangulation of a marked bordered surface. Any mutation-finite matrix with non-decomposable diagram admits an unfolding to a mutation-finite skew-symmetric matrix.
 \end{enumerate}
   According to the theorems above, if $G$ is the diagram associated to a skew-symmetrizable exchange matrix $M$, and $T$ is the ideal tagged triangulation corresponding to a particular $s$-decomposition $G_{dec}$ of $G$, then an unfolding to $G_{dec}$ defines a skew-symmetric diagram obtained by gluing of unfoldings of corresponding blocks. We design an algorithm that determines if a given graph is $s$-decomposable, and for each possible decomposition, finds the associated ideal tagged triangulation of bordered surface with marked points. In order to determine if a given skew-symmetrizable matrix has finite mutation type, it remains to check if it is of one of the 11 (for skew-symmetric) or 7 (for skew-symmetrizable) types. Since this requires a bounded number of operations, we obtain a linear algorithm. Moreover, this algorithm is linear in the size of the given matrix.

\section{Definitions}\label{definitions}
In this section, we introduce definitions  and a brief description of the algorithm. For convenience, we denote an edge that connects nodes $x,y$ by $\overline{xy}$ if the orientation of this edge is unknown or irrelevant, $\overrightarrow{xy}$ if the edge is directed from $x$ to $y$, and $\overleftarrow{xy}$ otherwise.
\begin{dfn}\label{gluingrules}
We recall that a diagram (or graph) is \textit{block-decomposable} (or \textit{decomposable}) if it is obtained by gluing elementary blocks of Table \ref{Decomposition} by the following \textit{gluing rules}:
\begin{enumerate}
\item Two white nodes of two different blocks can be identified. As a result, the graph becomes a union of two parts; the common node is colored black. A white node can neither be identified to itself nor with another node of the same block.
\item A black node can not be identified with any other node.
\item If two white nodes $x$, $y$ of one block (endpoints of edge $\overleftarrow{xy}$) are identified with two white nodes $p$, $q$ of another block (endpoints of edge $\overleftarrow{pq}$), $x$ with $p$, $y$ with $q$ correspondingly, then two parallel edges of the same direction is formed, and nodes $x=p$, $y=q$ are black.
\item If two white nodes $x$, $y$ of one block (endpoints of edge $\overleftarrow{xy}$) are identified with two white nodes $p$, $q$ of another block (endpoints of edge $\overleftarrow{pq}$), $x$ with $q$, $y$ with $p$ correspondingly, then both edges are removed after gluing, and nodes $x=q$, $y=p$ are black.
\end{enumerate}
\begin{table}[btch]
\centering
\begin{tabular}{cccccc}
&&&&&\\
\includegraphics{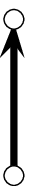}& \includegraphics{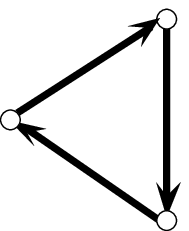}& \includegraphics{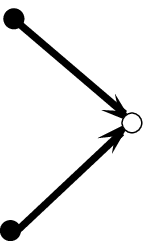} &\includegraphics{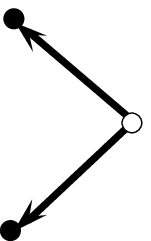} & \includegraphics{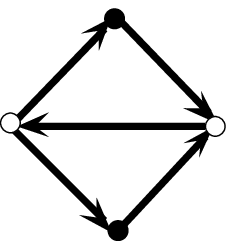}&\includegraphics{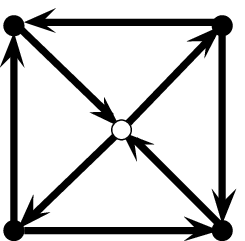}\\&&&&&\\
\textbf{Spike}&\textbf{Triangle}&\textbf{Infork}&\textbf{Outfork}&\textbf{Diamond}&\textbf{Square}
\end{tabular}
\caption{Elementary Blocks}\label{Decomposition}
\end{table}
\end{dfn}
\begin{dfn}
Let $B(G) = (b_{ij})$ be the skew-symmetric matrix whose rows and columns are labeled by the vertices of $G$, and whose entry $b_{ij}$ is equal to the number of edges going from $i$ to $j$ minus the number of edges going from $j$ to $i$. We say $B(G)$ is the \textit{adjacency matrix associated to $G$}, if an matrix $B=B(G)$, we say $G$ is the \textit{oriented adjacency graph associated to $B$}.
\end{dfn}
\begin{rmk}
By definition, the associated matrix to a oriented graph is skew-symmetric.
\end{rmk}

One of the applications of our algorithm involves mutations of cluster algebra, which requires the following definitions:

\begin{dfn}
A \textit{seed} is a pair $(f,B)$, where $f = {f_1,\ldots, f_n}$ form a collection of algebraically
independent rational functions of $n$ variables $x_1,\ldots, x_n$, and $B$ is a skew-symmetrizable
matrix. The part $f$ of seed $(f,B)$ is called \textit{cluster}, elements $f_i$ are called \textit{cluster variables},
and $B$ is called \textit{exchange matrix}.
\end{dfn}

\begin{dfn}
A quiver is a finite oriented multi-graph without loops and 2-cycles. Assume $B=(b_{ij})$ is the skew-symmetric matrix associated to a quiver $G$. We say that an $\overline{B}=(\overline{b}_{ij})$ is obtained from $B$ by \textit{matrix mutation} in direction $k$, and write $\overline{B}=\mu_k(B)$, if the entries of $\overline{B}$ are given by

\begin{equation*}
\overline{b}_{ij}=\begin{cases}
-b_{ij}, & \text{if $i=k$ or $j=k$,}\\
b_{ij}+\frac{1}{2}(|b_{ik}|b_{kj}+b_{ik}|b_{kj}|), & \text{otherwise.}
\end{cases}
\end{equation*}

Two matrices are called \textit{mutation-equivalent} if they can be transformed into each other by a sequence of mutations.
\end{dfn}

\begin{dfn}
A \textit{diagram} (or \textit{graph}) $S$ associated to a skew-symmetrizable integer matrix $B$ is an oriented graph with weighted edges obtained in the following way: Suppose $B=(b_{ij})_{i,j=1}^n$. Vertices of $S$ are labeled by $[1,\ldots, n]$. If $b_{ij} > 0$, we join vertices $i$ and $j$ by an edge directed from $i$ to $j$ and assign to this edge weight $-b_{ij}b_{ji}$.
\end{dfn}

It is shown in \cite{FZ2} that mutation of exchange matrices induce mutations of diagram. If $S$ is the diagram associated to $B$, and $B'=\mu_k{S}$ is a mutation of $B$ in direction $k$. We change the weigh in the way describe in Figure \ref{mutation}.

\begin{figure}
\centering
\input{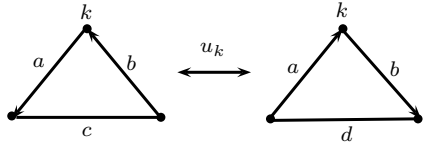}
\caption{Mutation of diagrams}\label{mutation}
\end{figure}

Here $\pm\sqrt{c}\pm\sqrt{d}=\sqrt{ab}$. The coefficient of $\sqrt{c}$ (resp. $\sqrt{d}$) is $1$ if the three edges form an oriented triangle and $-1$ otherwise. If $ab=0$, then neither value of $c$ nor orientation of the corresponding edges changes. Mutation class of an exchange matrix can be extended to mutation class of the associated diagram. Although the order of them may not be the same, it can be shown that mutation class of a matrix is finite if and only if a mutation class of the corresponding diagram is finite.

It is proved in \cite{FST1} that an skew-symmetric exchange matrix has finite mutation type if it is associated to a
decomposable graph or one of the 11 exceptional types. To generalize the results to skew-skymmetrizable matrix, we need the following definitions:

\begin{dfn}\label{unfoldingdef}
The \textit{unfolding} procedure is defined as follows (see section 4 in \cite{FST}): Suppose that we have a chosen disjoint index sets: $E_1,E_2,\ldots,E_n$, with $|E_i|=d_i$. Denote $m=\sum_{i=1}^nd_i$. To each matrix $B'$ mutation-equivalent to a given skew-symmetrizable $m\times m$ matrix $B$, a skew-symmetric matrix $\widehat{B}'$ indexed by $\bigcup_{i=1}^nE_i$ is defined so that the following conditions are satisfied:
\begin{enumerate}
\item the sum of entries in each column of each $E_i\times E_j$ block of $\widehat{B}'$ equals to $b_{ij}$;
\item if $b_{ij}\geq0$, then the $E_i\times E_j$ block of $\widehat{B}'$ has all entries non-negative.
\end{enumerate}
Define a composite mutation $\widehat{\mu}_i=\prod_{j\in E_i}\mu_j$ on $\widehat{B}'$. If $C$ is the skew-symmetric matrix constructed from $B$ satisfying the above conditions, we say $C$ is an unfolding if for any $B'$ mutation-equivalent to $B$, $\widehat{\mu}_i(\widehat{B}')=\widehat{\mu(B')}$.
\end{dfn}

\begin{dfn}\label{sdecom}
If a graph $G$ can be obtained by gluing both elementary blocks and new blocks in Table \ref{BOU} by the gluing rules in Definition \ref{gluingrules} and the following new rules, we say the graph is \textit{$s$-decomposable}:
\begin{enumerate}
\item If the graph has multiple edges containing $n$ parallel edges, replace the multiple edge by an edge of weight $2n$. For example, if we glue two parallel spikes of the same direction, we get an edge of weight 4 (see Figure \ref{ME4}).
    \begin{figure}[hctb]
    \centering
    \begin{minipage}[c]{0.5\linewidth}
    \centering
    \includegraphics[width=0.6\linewidth]{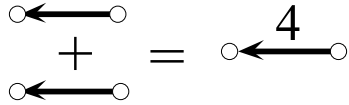}
    \caption{Edge of Weight 4}\label{ME4}
    \end{minipage}
    \end{figure}
\item All single edges have weight 1.
\end{enumerate}
\begin{table}
\begin{center}
\caption{Blocks of Unfolding}\label{BOU}
\begin{tabular}{cccc}
&&&\\
&{\large \textbf{New Blocks}}&{\large \textbf{Unfolding}}&{\large \textbf{Triangulation}}\\
\hline
{\large \textbf{Ia:}}&\input{lws}&\input{uflws}&\input{sflws}\\
{\large \textbf{Ib:}}&\input{rws}&\input{ufrws}&\input{sfrws}\\
{\large \textbf{II:}}&\input{tbt}&\input{uftbt}&\input{sftbt}\\
{\large \textbf{IIIa:}}&\input{lwd}&\input{uflwd}&\input{sflwd}\\
{\large \textbf{IIIb:}}&\input{rwd}&\input{ufrwd}&\input{sfrwd}\\
{\large \textbf{IV:}}&\input{twt}&\input{uftwt}&\input{sftwt}\\
{\large \textbf{V:}}&\input{sq}&\input{ufsq}&\input{sfsq}\\
\hline
\end{tabular}
\end{center}
\end{table}
\end{dfn}
\begin{rmk}
Suppose $G$ is associated to a skew-symmetric matrix. If $G$ is obtained by gluing blocks by the above rules, $G$ is associated to a ideal tagged triangulation of bordered surface with marked points obtained by the gluing of the pieces of surfaces associated to the blocks, two arcs are glued together iff the corresponding nodes are glued together in the associated blocks (see \cite{FST2}). Suppose $G$ is associated to a skew-symmetrizable matrix. If $G$ is s-decomposable, it is associated with triangulation of bordered surfaces with orbifold points. (see \cite{FST1})
\end{rmk}

%================10/25/2011================

\begin{rmk}\label{simple}
According to the above rules,  the weight of any edge in a decomposable graph can only be $1,2$ or $4$. All edges of weight $2$ can only be obtained from blocks in Table \ref{BOU}. Moreover, since all edges of weight 2 contains at least one black endpoint, we can never obtain an edge of weight 4 from edges of weights 2. Moreover, an edge of weight 4 can only be obtained from \twt or by Figure \ref{ME4}.
\end{rmk}

  A geometric interpretation of mutations on the blocks in Table \ref{BOU} is given in Lusztig's \cite{L}.
Abusing the notation, we say the new blocks are the \textit{foldings} of their corresponding unfoldings. Each unfolding represents an ideal tagged triangulation of bordered surface with marked points (see pictures in \cite{FST}, (Table 7.1)). Each of these unfoldings except the last one corresponds to triangulations with two conjugate edges inside a digon (or monogon). Conjugate edges represent the same vertex in the foldings. Mutation of the folded vertex corresponds to the flips of both edges in the conjugate pair. \textit{Composite flip} of the triangulation corresponding to an unfolding diagram is defined as a collection of flips in all edges that represent vertices in the set $E_i$. Note that any two flips in a composite flip commute, see Figure \ref{flip}.

\begin{figure}[bcth]
\centering
\input{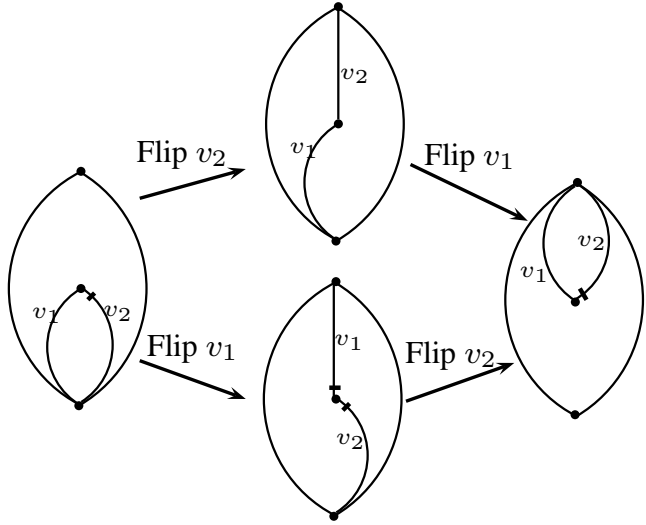}
\caption{Composite Flip}\label{flip}
\end{figure}

%============10/25/2011================

  In \cite{FST}, the following theorem is proved:
\begin{thm}\label{thm_unfolding}
Any s-decomposable diagram admits an unfolding to a diagram arising from ideal tagged triangulation of a marked bordered surface. Any mutation-finite
matrix with non-decomposable diagram admits an unfolding to a mutation-finite skew-symmetric matrix.
\end{thm}

  Given an $s$-decomposable diagram with a fixed decomposition, there is a unique tagged triangulation of a marked bordered surfaces with chosen tuples of conjugate pairs of edges. This surface can be obtained by gluing pieces of surfaces representing unfoldings of corresponding blocks in the decomposition along edges corresponding to glued white vertices. The construction is invariant under mutation: mutating the diagram means performing composite flips to the original triangulations. Furthermore, the following theorems are proved in \cite{FST}:
\begin{thm}\label{thm_decomposable}
There is a one-to-one correspondence between s-decomposable skew-symmetrizable diagrams with fixed block decomposition and ideal tagged triangulations of marked bordered surfaces with fixed tuple of conjugate pairs of edges.
\end{thm}
\begin{thm}\label{thm_mutation_finite}
A skew-symmetrizable $n\times n$ matrix, $n\geq3$, that is not skew-symmetric,
has finite mutation class if and only if its diagram is either $s$-decomposable or mutation-
equivalent to one of seven types.
\end{thm}

  By the previous theorems, to check that a given skew-symmetrizable non skew-symmetric matrix has finite mutation type, first it only takes finitely many operations to check if it is mutation-equivalent to one of the seven exceptional types of diagrams. If not, we can further check if the associated adjacency graph $G$ is $s$-decomposable. In below sections, we develop an algorithm linear in size of $G$ that determines if a given diagram is $s$-decomposable.

  For convenience, we need the following definition:
\begin{dfn}\label{nbhd}
   Suppose $N$ is a subgraph of $G$ with all its nodes colored white or black. If there exists another quiver $M$ with all its nodes colored white or black, such that $G$ can be obtained by gluing $M$ to $N$ by the rules in Definition \ref{gluingrules} and \ref{sdecom}, we say $N$ is a \textit{colored subgraph} of $G$. A \textit{neighborhood} of $o$ is a colored subgraph of $G$ that contains node $o$. We say a colored subgraph $N$ of $G$ is \textit{decomposable} if there exists an $s$-decomposable or block-decomposable graph $\widetilde{G}$ that contains $N$ as a colored subgraph. A colored subgraph $N$ of $G$ is said to be \textit{indecomposable} if any graph that contains $N$ as a colored subgraph is neither $s$-decomposable nor decomposable. We say a colored subgraph $N$ is \textit{$s$-decomposable as a subgraph} if $N$ can be obtained by gluing blocks by the rules in Definition \ref{gluingrules} and \ref{sdecom}, and the resulting color of nodes in $N$ coincides with the original color of vertices in $N$.
\end{dfn}
\begin{rmk}\label{mnbhd}
Note that if for a vertex $o$ of graph $G$, the whole graph $G$ is obtained by gluing a colored subgraph to a neighborhood $N$ of $o$, no edge in $N$ can be annihilated by gluing procedure. For a given graph $G$ and a selected node $o$, the set of neighborhoods of $o$ in $G$, denoted by $\mathcal{N}_o$ is
a partially ordered set by inclusion. We define three subsets of $\mathcal{N}_o$ as follows:
\begin{itemize}
\item $\mathcal{I}_o$ is the set of all decomposable neighborhoods each of which contains all edges incident to $o$.
\item $\mathcal{D}_o$ is the set of all decomposable neighborhoods of $o$ each of which is decomposable as a subgraph.
\item $\mathcal{S}_o=\{N\subset\mathcal{I}_o\cap\mathcal{D}_o\mbox{ }|\mbox{ $N$ is minimal}\}$.
\end{itemize}

  If $\mathcal{S}_o$ is empty, then by definition, the graph containing $o$ is not $s$-decomposable.
\end{rmk}

  Our goal is to find a combinatorial algorithm which determines if a given graph is $s$-decomposable.
%Moreover, we want to retrieve the blocks that are used to construct the given decomposable graph.
According to remark \ref{simple}, we can determine if a graph contains blocks from Table \ref{BOU} by locating edges of weight 2 and analyzing their neighborhoods. We differ cases by the number of edges of weight $2$ that are incident to the considered node $o$. Denote this number by $n$. According to the rules of gluing and Table \ref{BOU}, $n$ is at most $4$ for any node in an $s$-decomposable graph. If none of the edges incident to node $o$ has weight two, $o$ can not be contained in a block from Table \ref{BOU} and we skip node $o$. Therefore, $n=1,2,3$ or $4$. Starting with any node $o$ with $n=4$, we check if $\mathcal{S}_o$ is non-empty by examining the following information that can be directly observed from the graph: degree of $o$, degree of the nodes that are connected to $o$ by one edge, and the number and directions of the edges between node $o$ and the nodes connected to $o$. If $\mathcal{S}_o$ is empty, the graph is not $s$-decomposable. If $o$ is contained in a decomposable neighborhood in $\mathcal{S}_o$, we replace the neighborhood by another one which is \textit{consistent} in the sense that the new graph is $s$-decomposable if and only if the original one is. The new neighborhood does not contain any edge of weight 2. After all nodes with $n=4$ are exhausted, we proceed to the nodes with $n=3,2,1$ (in decreasing order). Finally, we get a graph that contains only edges of weight $1$ and $4$. The new graph is $s$-decomposable if and only if it is decomposable (see \cite{WG}). Then we apply the algorithm from \cite{WG} to determine it the graph is decomposable. Since all replacements are consistent, we can determine if the original graph is $s$-decomposable.

\section{Algorithm}\label{algorithm}
In \cite{WG}, it is proved that we can assume that the graph is connected when only blocks in Table \ref{Decomposition} are used. If the graph is $s$-decomposable, it is easy to see that we can make the same assumption as well. In fact, except \tbt, none of the edges can be annihilated by gluing a block from Table \ref{BOU} to any graph. If \tbt is glued to an existing graph, causing $\overline{uw}$ to be annihilated, nodes $u,v$ are still connected via $\overline{uv}$ and $\overline{vw}$. Therefore, gluing new blocks will not break connectivity.

Let $\mathcal{U}$ be the collection of both old and new blocks.
In order to find $\mathcal{S}_o$ of a node $o$, it suffices to check if node $o$ has a neighborhood that is isomorphic to some graph in $\mathcal{U}$ or obtained by gluing two blocks from $\mathcal{U}$. By checking all nodes in $\mathcal{U}$, we analyze the results in Table \ref{n=4}-\ref{m=3}.

If any of the neighborhoods in $\mathcal{U}$ is a disjoint connected component, the algorithm stops. If the neighborhood may not be a disjoint connected component (DCC), we apply suitable replacement as suggested in the tables below. Note that we need those replacements to be consistent, i.e. the original graph is $s$-decomposable if and only if the new graph is. The consistency of all replacements can be checked by exhausting analysis of all neighborhoods in $\mathcal{U}$ and Lemma 1 in \cite{WG}.
\subsection{$n=4$}
We have the following two situations:
\begin{longtable}{|c|c|c|}\hline
& \textbf{A} & \textbf{B} \\\hline
Decomposition  &  \includegraphics{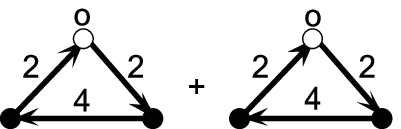}  &  \includegraphics{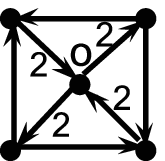}\\\hline
Degree of $o$  &  4                            &   4                         \\\hline
Replacement       & DCC &  DCC                  \\\hline
\caption{$n=4$}\label{n=4}
\end{longtable}

\subsection{n=3}
We have the following eight situations:
\pagebreak
\begin{longtable}{|c|c|c|}\hline
& \textbf{A1} & \textbf{A2} \\\hline
Decomposition  &  \includegraphics{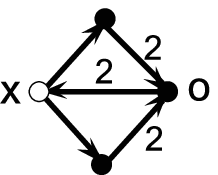}  &  \includegraphics{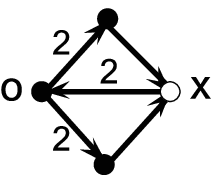}\\\hline
Degree of $o$  &  3                            &   3                         \\\hline
Replacement       &  \includegraphics{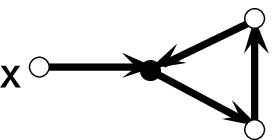} &  \includegraphics{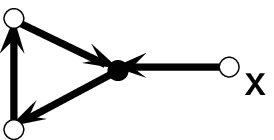}\\\hline\hline
& \textbf{B1} & \textbf{B2} \\\hline
Decomposition  &  \includegraphics{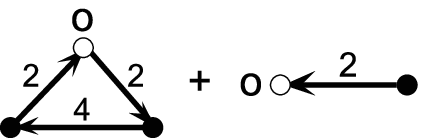}  &  \includegraphics{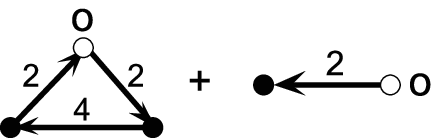}\\\hline
Degree of $o$  &  3                            &   3                         \\\hline
Replacement       &         DCC                   &          DCC                \\\hline\hline
& \textbf{C1} & \textbf{C2} \\\hline
Decomposition  &  \includegraphics{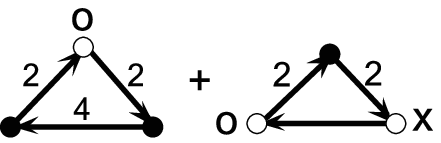}  &  \includegraphics{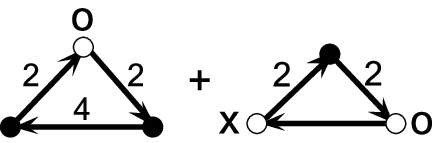}\\\hline
Degree of $o$  &  4                            &   4                         \\\hline
Replacement       &  \includegraphics{rpright.eps}& \includegraphics{rpleft.eps}\\\hline\hline
& \textbf{D1} & \textbf{D2} \\\hline
Decomposition  &  \includegraphics{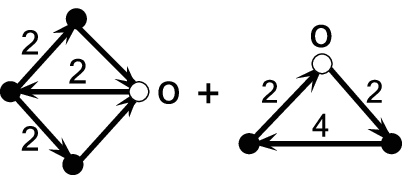}  &  \includegraphics{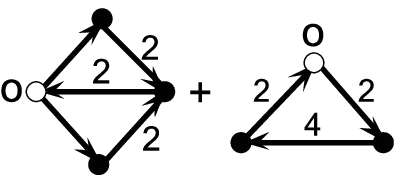}\\\hline
Degree of $o$  &  5                            &   5                         \\\hline
Replacement       &         DCC                   &          DCC                \\\hline

\caption{n=3}\label{n=3}
\end{longtable}
In this case the degree of the considered node can only be $3,4$ or $5$, otherwise the graph is not $s$-decomposable. For a given graph $G$, to determine neighborhood $o$ of what type is considered, we examine the degree of $o$.

First suppose the degree of $o$ is 3. We only need to consider \textbf{A1,A2,B1,B2}. We call the nodes connected to $o$ by an edge \textit{boundary nodes}. If one of the boundary nodes, denoted by $x$, is connected to the remaining two by edges of weight 1, $o$ can only be contained in \textbf{A1} or \textbf{A2}. Note that in either cases, the degree of $x$ is no less than 3, and the degrees of the remaining two boundary nodes are 2. If only two of the boundary nodes are connected, $o$ can only be contained in \textbf{B1} or \textbf{B2}. In both cases, the neighborhoods are disjoint connected components.

Second, suppose the degree of $o$ is 4. Node $o$ can only be contained in \textbf{C1,C2}, otherwise the graph is not $s$-decomposable. Note that in this case, one boundary node is connected to $o$ by an edge of weight 1. Denote this node by $x$. The remaining three boundary nodes are connected to $o$ by edges of weight 2, two of them are connected by an edge of weight 4, the third one is connected to $x$ by an edge of weight 2. Moreover, the degree of $x$ is no less than 2, the remaining boundary nodes all have degree 2.

Finally, suppose the degree of $o$ is 5. In this case, $o$ can only be contained in \textbf{D1,D2}, otherwise the graph is not $s$-decomposable. In either case, the neighborhood is a disjoint connected component.

\subsection{$n=2$}
We have the following 13 situations:

\begin{longtable}{|c|c|c|}\hline
& \textbf{A} & \textbf{B} \\\hline
Decomposition & \input{n=2_1.tex} &\includegraphics{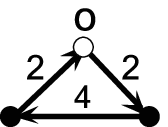}\\\hline
Degree of $o$ & 2                           &  $\geq2$                            \\\hline
Replacement      & \includegraphics{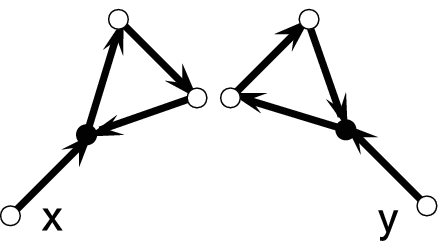}&\includegraphics{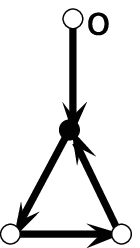}\\\hline\hline
& \multicolumn{2}{|c|}{\textbf{C}}\\\hline
&\multicolumn{2}{|c|}{}\\
Decomposition & \multicolumn{2}{|c|}{\input{n=2_3.tex}\input{n=2_3_1.tex}\input{n=2_3_2.tex}}\\&\multicolumn{2}{|c|}{}\\\hline
Degree of $o$ & \multicolumn{2}{|c|}{2}\\\hline
Replacement   & \multicolumn{2}{|c|}{DCC}\\\hline\pagebreak\hline
& \textbf{D1} & \textbf{D2} \\\hline
Decomposition & \input{n=2_4.tex} &\input{n=2_5.tex}\\\hline
Degree of $o$ & 3                           &  3                            \\\hline
Replacement      & \includegraphics{rpright.eps}&\includegraphics{rpright.eps}\\\hline\hline
& \textbf{E1} & \textbf{E2} \\\hline
Decomposition & \input{n=2_6.tex} &\input{n=2_7.tex}\\\hline
Degree of $o$ & 4                           &  4                            \\\hline
Replacement      & DCC                         &  DCC                          \\\hline
& \textbf{E3} & \textbf{E4} \\\hline
Decomposition & \input{n=2_8.tex} &\input{n=2_9.tex}\\\hline
Degree of $o$ & 4                           &  4                            \\\hline
Replacement      & DCC                         &  DCC                          \\\hline\hline
& \textbf{F1} & \textbf{F2} \\\hline
Decomposition & \input{n=2_10.tex} &\input{n=2_11.tex}\\\hline
Degree of $o$ & 5                           &  5                            \\\hline
Replacement      &  \includegraphics{rpright.eps} & \includegraphics{rpleft.eps} \\\hline\pagebreak\hline
& \textbf{F3} & \textbf{F4} \\\hline
Decomposition & \input{n=2_12.tex} &\input{n=2_13.tex}\\\hline
Degree of $o$ & 5                           &  5                            \\\hline
Replacement   &  \includegraphics{rpright.eps} & \includegraphics{rpleft.eps}\\\hline

\caption{n=2}\label{n=2}
\end{longtable}
To determine if the considered node $o$ is contained in any of the neighborhoods in Table \ref{n=2}, we consider the degree of $o$. Note that when $n=2$, the degree of node $o$ in Table \ref{n=2} takes only $2,3,4$ or $5$, otherwise, the graph is not $s$-decomposable.

If the degree of $o$ is 2, $o$ in $s$-decomposable graph can have only neighborhoods of type \textbf{A,B,C}, (see Table \ref{n=2}.) In this case we denote the other endpoints of the edges of weight 2 by $x,y$.
\begin{itemize}
\item If $x,y$ are not connected by an edge, the graph is $s$-decomposable in two cases. First, the neighborhood can be obtained by gluing \{\lws,\rws\}, or two \lws, or two \rws, as shown in case \textbf{C}; second, the neighborhood can be obtained from annihilating edge $\overleftarrow{xy}$ in \tbt, as shown in case \textbf{A}. To determine how the neighborhood is obtained, note that in the first case, the neighborhood is a disjoint connected component, and in the second case, the edge $\overleftarrow{xy}$ can be annihilated by an edge from a spike, a triangle or the mid-edge of a diamond. Suppose the degrees of nodes $x,y$ are both 1. If edge $xo$ and $yo$ are both directed away from or towards node $o$, then the neighborhood is obtained in the way shown in the second or third picture in case \textbf{C}, otherwise the graph is not $s$-decomposable. Suppose two edges have different orientations, there are only two cases when the graph is $s$-decompositions. If the degree of nodes $x,y$ are both greater than 1, the graph is decomposable only if the neighborhood is obtained by gluing another block to \tbt. Then we can apply the corresponding replacement as shown in case \textbf{A}. If the degree of nodes $x,y$ are both 1, the neighborhood is a DCC and can be obtained either by annihilating $\overline{xy}$ in \tbt by gluing a spike, or by gluing \{\lws,\rws\}.
\item If $x,y$ are connected by an edge of weight 4 from $y$ to $x$, there are two cases. First, the neighborhood can be obtained from gluing an edge $\overrightarrow{yx}$ to the graph as in case \textbf{A}. In this situation, the degrees of nodes $x,y$ are at least 2; Second, the neighborhood can be obtained from case \textbf{B}. Therefore, to distinguish the above two cases, we first check the degrees of $o$. If $o$ has degree greater than 2, $o$ is contained in a neighborhood as shown in case \textbf{B}. If $o$ has degree 2, we check the degrees of $x,y$: if the degrees of both $x,y$ are two, the neighborhood is a disjoint connected component and has two possible decompositions; if the degrees of both $x,y$ exceed 2, $o$ must be contained in a neighborhood in \textbf{A}. In latter case, after applying the corresponding replacement, we keep edge $\overline{xy}$ and change its weight from 4 to 1.
\item If $x,y$ are connected by an edge of weight one from $y$ to $x$, the neighborhood can only be obtained from case \textbf{A}.
\end{itemize}

Next suppose the degree of $o$ is 3. In this case node $o$ can only be contained in a neighborhood shown in \textbf{B,D1} or \textbf{D2}, otherwise the graph is not $s$-decomposable. To distinguish these cases, first we check if the nodes connected to $o$ by an edge of weight 2 are connected by an edge of weight 4. If so, $o$ must be contained in \textbf{B}. If not, we denote the three nodes connected to $o$ by $x,y,z$, where $y,z$ are connected to $o$ by edges with weight 2. Note that $x$ must be connected to $o$ via an edge with weight 1, $x,y$ must be connected via an edge $\overrightarrow{yx}$ with weight 2, degrees of $y$ and $z$ must be 2 otherwise the graph is not $s$-decomposable. Node $o$ is contained in a neighborhood shown as in \textbf{D1} if $o,z$ are connected via an edge $\overrightarrow{oz}$, \textbf{D2} if via $\overrightarrow{za}$.

Next suppose the degree of $o$ is 4. In this case node $o$ can only be contained in a neighborhood shown in one of \textbf{B}, \textbf{E1}-\textbf{E4}, otherwise the graph is non $s$-decomposable. By the same argument as in this previous case, we check if $o$ is contained in \textbf{B}. If not, we need to determine if $o$ is contained in any of \textbf{E1}-\textbf{E4}. Denote the nodes that are connected to $o$ by edges with weight 1 by $y,z$, the nodes that are connected to $o$ by edges with weight 2 by $x,w$. Then $y,z$ must both be connected to one of the nodes that are connected to $o$. Assume $y,z$ are both connected to $x$, then $\overline{xy},\overline{xz}$ both have weight 2. By picture \textbf{E1}-\textbf{E4}, $\overline{ow}$ must have weight 2, and the degree of $w$ is 2. By examination of the orientation of the edges incident to $o$ we determine in with type of neighborhood $o$ is contained.

Finally, if the degree of $o$ is 5, it can be only contained in a neighborhood shown in one of \textbf{B}, \textbf{F1}-\textbf{F4}, otherwise the graph is non $s$-decomposable. As above, we check if $o$ is contained in \textbf{B}. If not, we denote the boundary nodes of $o$ by $x,y,z,u,v$, where $x,y,z$ are connected to $o$ by edges of weight 1, $u,v$ are connected to $o$ by edges of weight 2. Note that if the graph is $s$-decomposable, one of $u,v$ must be connected to two nodes among $x,y,z$ by edges of weight 2. Assume $u$ is connected to $y$ and $w$ by edges of weight 2. Then, $v$ must be connected to $x$ by another edge of weight 2, and $\mbox{deg}(y)=\mbox{deg}(w)=\mbox{deg}(v)=2$, $\mbox{deg}(u)=3$, $\mbox{deg}(x)\geq2$.

\subsection{$n=1$}
In this case, there is only one edge with weight $2$ that is incident to $o$. Denote the other endpoint of this edge by $p$. We consider the number of edges with weight $2$ incident to $p$. Denote this number by $m$.

If $m=1$, there are two cases, as shown Table. \ref{m=1}. We can only attach blocks containing no edge of weight two to the node $p$. In both cases, degree of $o$ is one. It is easy to determine if $o$ is contained in \textbf{A1} or \textbf{A2}.\pagebreak

\begin{longtable}{|c|c|c|}\hline
& \textbf{A1} & \textbf{A2} \\\hline
Decomposition  &  \includegraphics{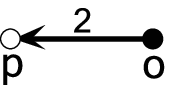}  &  \includegraphics{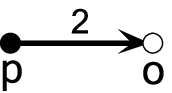}\\\hline
Degree of $p$  &  $\geq1$                            & $\geq1$                         \\\hline
Replacement    &  \includegraphics{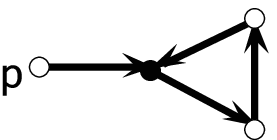}& \includegraphics{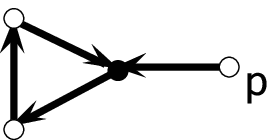}\\\hline

\caption{$m=1$}\label{m=1}
\end{longtable}

If $m=2$, there are ten possible cases, as shown in Table \ref{m=2}. In either of the cases \textbf{A1,A2}, we can only attach to $o$ blocks containing no edge of weight 2. Hence after applying the corresponding replacement, there is no edge with weight 2 that is incident to $o$. In case \textbf{B1} or \textbf{B2}, we can only attach to $p$ blocks containing no edge of weight 2.

\begin{longtable}{|c|c|c|}\hline
& \textbf{A1} & \textbf{A2} \\\hline
Decomposition  &  \input{m=2_1.tex}    &  \input{m=2_2.tex}\\&&\\\hline
Degree of $p$  &  $2$                            & $2$                         \\\hline
Replacement    &  \includegraphics{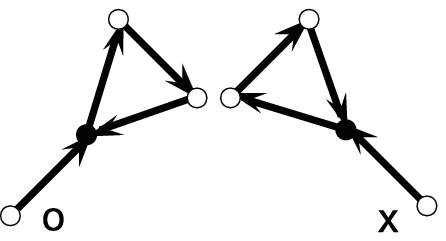}  & \includegraphics{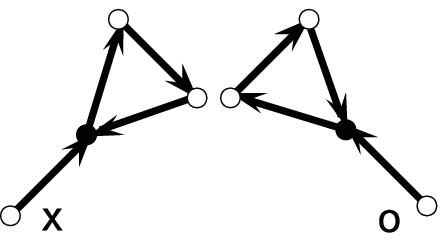}\\\hline\hline
& \textbf{B1} & \textbf{B2} \\\hline
Decomposition  &  \input{m=2_3.tex}    &  \input{m=2_4.tex}\\\hline
Degree of $p$  &  $\geq2$                        & $\geq2$                         \\\hline
Replacement    &  \includegraphics{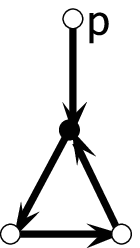}    &  \includegraphics{rpm=2_3_4}\\\hline\pagebreak\hline
& \textbf{C1} & \textbf{C2} \\\hline
Decomposition  &  \input{m=2_5.tex}  &  \input{m=2_6.tex}\\\hline
Degree of $p$  &  $2$                          & $2$                         \\\hline
Replacement       &  DCC                          & DCC                          \\\hline\hline
& \textbf{D1} & \textbf{D2} \\\hline
Decomposition  &  \input{m=2_7.tex}  &  \input{m=2_8.tex}\\\hline
Degree of $p$  &  $3$                          & $3$                         \\\hline
Replacement    &  \input{rreplacey.tex}&  \input{rreplacey.tex}\\\hline\hline
& \textbf{D3} & \textbf{D4} \\\hline
Decomposition  &  \input{m=2_9.tex}  &  \input{m=2_10.tex}\\\hline
Degree of $p$  &  $3$                          & $3$                         \\\hline
Replacement    &  \input{rreplacey.tex}& \input{rreplacey.tex}\\\hline

\caption{$m=2$}\label{m=2}
\end{longtable}

Table \ref{m=2} gives all possible cases with $m=2$. To determine the type of neighborhood $o$ is contained in, let us denote the node connected to $o$ by an edge of weight 2 by $p$, then examine the degree of $p$. According to Table \ref{m=2}, if the graph is $s$-decomposable, $\mbox{deg}(p)\geq2$.

Suppose the degree of $p$ is 2, we denote the other node that is connected to $p$ by $x$. Note that the weight of $\overline{px}$ must be 2. If $x$ is connected to $o$ by an edge with weight 4, then $o$ must be contained in neighborhood \textbf{B1} or \textbf{B2} depending on the orientation of edges. Note that in this case, the graph is a disjoint connected component. If $\overline{xo}$ has weight 1, then $o$ is contained in neighborhood \textbf{A1} or \textbf{A2} depending on the orientation of edges. If $x$ is not connected to $o$, then the graph must be a disjoint connected component \textbf{C1} or \textbf{C2}.

Next, suppose the degree of $p$ is 3. By Table \ref{m=2}, there are two edges with weight 2 that are incident to $p$, one of which is $\overline{op}$. Denote the other edge of weight 2 by $\overline{ox}$ (the other endpoint is $x$). If $x$ is connected to $o$ by an edge of weight 4, then $o$ is contained in neighborhood \textbf{B1} or \textbf{B2}. If $x$ is not connected to $o$, then $o$ lies in neighborhood \textbf{D1}-\textbf{D4}. Note that in the latter case, the degree of $o$ is 1. If it is neither of the above two situation, the graph is not $s$-decomposable.

Finally if the degree of $p$ is greater than 3, $o$ must be contained in neighborhood \textbf{B1} or \textbf{B2}.

%============================1/28/2011===================================
If $m=3$, there are fourteen cases, as shown in Table.\ref{m=3}.

\begin{longtable}{|c|c|c|}\hline
& \textbf{A1} & \textbf{A2} \\\hline
Decomposition  &  \includegraphics{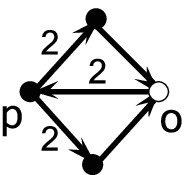}  &  \includegraphics{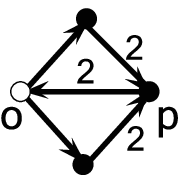}\\\hline
Degree of $p$  &  $3$                           & $3$                         \\\hline
Replacement    &  \includegraphics{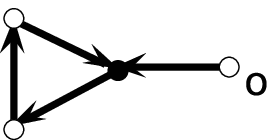}& \includegraphics{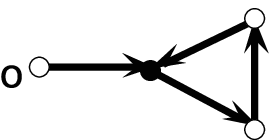}\\\hline\hline
& \textbf{B1} & \textbf{B2} \\\hline
Decomposition  &  \input{m=3_3.tex}  &  \input{m=3_4.tex}\\&&\\\hline
Degree of $p$  &  $3$                            & $3$                         \\\hline
Replacement       &  DCC                            & DCC                         \\\hline\hline
& \textbf{B3} & \textbf{B4} \\\hline
Decomposition  &  \input{m=3_5.tex}  &  \input{m=3_6.tex}\\&&\\\hline
Degree of $p$  &  $3$                          & $3$                         \\\hline
Replacement       & DCC                           & DCC                         \\\hline\hline
& \textbf{C1} & \textbf{C2} \\\hline
Decomposition  &  \includegraphics{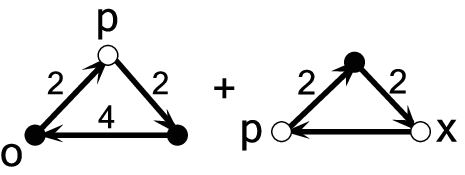}  &  \includegraphics{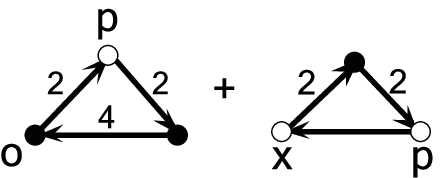}\\\hline
Degree of $p$  &  $4$                            & $4$                         \\\hline
Replacement       & \includegraphics{rpright.eps}& \includegraphics{rpleft.eps}\\\hline\pagebreak\hline
& \textbf{C3} & \textbf{C4} \\\hline
Decomposition  &  \includegraphics{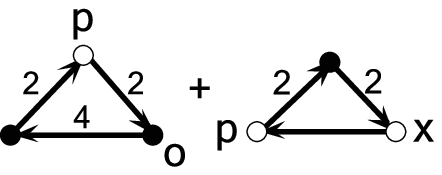}  &  \includegraphics{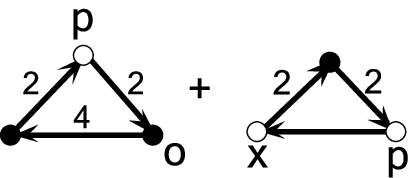}\\\hline
Degree of $p$  &  $4$                            & $4$                         \\\hline
Replacement       & \includegraphics{rpright.eps}& \includegraphics{rpleft.eps}\\\hline\hline
& \textbf{D1} & \textbf{D2} \\\hline
Decomposition  &  \input{m=3_11.tex}  &  \input{m=3_12.tex}\\&&\\\hline
Degree of $p$  &  $5$                           & $5$                          \\\hline
Replacement       &  DCC                           & DCC                          \\\hline\hline
& \textbf{D3} & \textbf{D4} \\\hline
Decomposition  &  \input{m=3_13.tex}  &  \input{m=3_14.tex}\\\hline
Degree of $p$  &  $5$                           & $5$                         \\\hline
Replacement       &  DCC                           & DCC                         \\\hline

\caption{m=3}\label{m=3}
\end{longtable}

Note that the degree of node $o$ in all pictures is 2 except \textbf{A1} and \textbf{A2}. Therefore, if the considered node has degree larger than 2, it can only be contained in neighborhood \textbf{A1} or \textbf{A2}. In both cases, there are two nodes, denoted by $x,y$, that are connected to $o$ by edges of weight 1, and $p$ is connected to both nodes $x$ and $y$ by edges of weight 2. Moreover, the degree of $p$ is 3, the degrees of $x,y$ are both 2. Suppose $o$ has degree 2, according to Table \ref{m=3}, $\mbox{deg}(p)=3,4$ or $5$.

First, suppose $\mbox{deg}(p)=3$. $o$ can only be contained in a neighborhood shown in \textbf{B1,B2,B3} or \textbf{B4}. Note that in all these cases, all edges incident to $p$ has weight 2, and the graph is a disjoint connected component.

Next, suppose $\mbox{deg}(p)=4$. Then $o$ can only be contained in neighborhoods of type \textbf{C1,C2,C3} or \textbf{C4}. In all these cases, $p$ is incident to four edges of weight 2. Also $p$ is connected to a node which is also connected to $o$ by an edge of weight 4. Among the four nodes connected to $p$, three of them, including $o$, have degree 2, the remaining node has degree no less than 2. We can check the orientations of all edges to determine which neighborhood $o$ is contained in.

Finally, suppose $\mbox{deg}(p)=5$. In this case, $o$ can only be contained in neighborhood \textbf{D1}-\textbf{D4}, and the graph is a disjoint connected component. In all these cases, $p$ is incident to three edges of weight 2. Denote the other endpoints of these edges besides $o$ by $x,y$. Node $p$ is also incident to two edges of weight 1. Denote the other endpoints of these two edges by $z,w$. According to Table \ref{m=3}, $z,w$ must both be connected to one of $x,y$ by edges of weight 2. Assume it is $y$. Then $o$ is connected to $x$ by an edge of weight 4. Note that in this case $\mbox{deg}(y)=3$, $\mbox{deg}(z)=\mbox{deg}(w)=\mbox{deg}(x)=2$. We can determine which neighborhood $o$ is contained in by examining the orientations of the edges.

\section{Summary}

In Section \ref{algorithm}, we exhausted all nodes that are incident to some edges with weight 2. We also replace a neighborhood of any such node by a consistent one which does not contain any edge with weight 2. Therefore, for any given weighted graph, we can determine if it is $s$-decomposable, and simplify it into a graph containing only edges with weight 1 or 4. Then we apply the algorithm in \cite{WG} to determine if it is block decomposable. Note that every node is examined at most twice: once in the procedure as in \ref{algorithm}, once in the algorithm in \cite{WG}. Hence the algorithm is linear in the size of the given graph.

Apply the algorithm to Theorem \ref{thm_mutation_finite}, we get the following corollary:

\begin{crl}
Given a skew-symmetrizable matrix $B$, there exists an algorithm linear in the size of $B$ to determine if $B$ has finite mutation type.
\end{crl}
\begin{proof}
Assume the size of $B$ is no less than 3. First, we check if $B$ is mutation-equivalent to one of the seven exceptional types in Theorem \ref{thm_mutation_finite}. If so, $B$ is mutation finite. Since the sizes of all seven types do not exceed 6, it only takes finite number of operation. If none of the seven types is mutation equivalent to $B$, we apply our algorithm to the associated adjacency graph of $B$. By the previous argument, the number of operation it requires is linear in the size of $B$. If the adjacency graph is confirmed to be $s$-decomposable, $B$ has finite mutation type.
\end{proof}

\begin{rmk}
If diagram $G$ is $s$-decomposable, our algorithm can recover the blocks used to obtain $G$ since every step of replacement is consistent. In particular, we can determine the ideal tagged triangulation of bordered surfaces with marked points to each decomposition.
\end{rmk}
\begin{rmk}
A connected diagram $G$ has non-unique decomposition if and only if $G$ is isomorphic to one of the two diagrams in Figure \ref{nu}.
\end{rmk}
\begin{figure}[bcth]
\renewcommand{\captionlabeldelim}{}
\centering
\input{nonunique1.tex}
\input{nonunique2.tex}
\caption{}\label{nu}
\end{figure}
\bibliography{unfolding}{}
\bibliographystyle{unsrt}

\end{document}